\documentclass{article}
\usepackage{amsmath, amsthm, epsfig,amssymb, multicol, tikz}
\usepackage{hyperref}

\newtheorem{lemma}{Lemma}
\newtheorem{theorem}{Theorem}

\newtheorem{conjecture}{Conjecture}
\newtheorem{definition}{Definition}

\usepackage{bm}

\usepackage{rotating}
\usepackage{tikz}
\usetikzlibrary{topaths,calc}

\tikzstyle{vertex} = [fill,shape=circle,node distance=80pt]
\tikzstyle{edge} = [fill,opacity=.5,fill opacity=.5,line cap=round, line join=round, line width=50pt]
\tikzstyle{elabel} =  [fill,shape=circle,node distance=30pt]

\pgfdeclarelayer{background}
\pgfsetlayers{background,main}

\title{The maximum $p$-Spectral Radius of Hypergraphs with $m$ Edges}

\author{
Linyuan Lu
\thanks{University of South Carolina, Columbia, SC 29208, 
({\tt lu@math.sc.edu}). This author was supported in part by NSF
grant DMS 1600811 and ONR grant N00014-17-1-2842.}
}
\begin{document}
\maketitle
\begin{abstract}
 For $r\geq 2$ and $p\geq 1$, the $p$-spectral radius of an $r$-uniform hypergraph $H=(V,E)$ on $n$ vertices 
  is defined to  be
  $$\rho_p(H)=\max_{{\bf x}\in \mathbb{R}^n: \|{\bf x}\|_p=1}r \cdot \!\!\!\!
  \sum_{\{i_1,i_2,\ldots, i_r\}\in E(H)}
    x_{i_1}x_{i_2}\cdots x_{i_r},$$
    where the maximum is taken over all ${\bf x\in \mathbb{R}^n}$ with the $p$-norm equals 1.
    In this paper, we proved for any integer $r\geq 2$, and any real $p\geq 1$,
    and any $r$-uniform hypergraph $H$ with $m={s\choose r}$ edges (for some real $s\geq r-1$), we have $$\lambda_p(H)\leq \frac{rm}{s^{r/p}}.$$
    The equality holds if and only if $s$ is an integer and $H$ is the complete $r$-uniform hypergraph $K^r_s$ with some possible isolated vertices added. Thus, we completely settled a conjecture of Nikiforov.
In particular, we settled all the principal cases of the Frankl-F\"{u}redi's Conjecture on the Lagrangians of $r$-uniform hypergraphs for all $r\geq 2$.
\end{abstract}


\textsl{keywords:} $p$-spectral radius, uniform hypergraph, adjacency tensor,
Lagrangian, Frankl-F\"uredi Conjecture\\

\section{History}
For $r\geq 2$, an $r$-uniform hypergraph $H$ on $n$ vertices consists of a vertex set $V$ and an edge set  $E\subseteq {V\choose r}$. Cooper and Dutle \cite{Cooper} defined
the adjacency tensor $A$ of  $H$ to be the  $r$-order $n$-dimensional tensor
$A=(a_{i_1 \cdots i_r})$ by
$$a_{i_1 \cdots i_r}=
 \begin{cases}
\frac{1}{(r-1)!}  & \text{if $\{i_1, \ldots,  i_r\}$ is an edge of $H$ ,} \\
0  & \text{otherwise,}
\end{cases}
$$
where each $i_j$ runs from $1$ to $n$ for $j\in[r]$. 
The adjacency tensor $A$ of $r$-uniform hypergraph is always  nonnegative and symmetric.
  
Given an $r$-uniform hypergraph $H$ (on $n$ vertices), the polynomial form $P_H(\mathbf{x})\colon \mathbb{R}^n\to \mathbb{R}$  is defined for any vector  $\mathbf{x}=\left( x_1, \ldots, x_n\right) \in \mathbb{R}^n$ as 
$$ P_H(\mathbf{x})= \sum\limits_{i_1,  \ldots,  i_r=1}^n a_{i_1 \cdots i_r} x_{i_1} \cdots x_{i_r}= r\sum\limits_{\{i_1,  \ldots,  i_r\}\in E(H)}x_{i_1} \cdots x_{i_r}.$$
For $p\geq 1$, the {\em $p$-norm} of a vector $\bm{x}\in\mathbb{R}^n$ is
$$\Vert\mathbf{x}\Vert_p=\left(\sum\limits_{i=1}^{n} |x_i|^p\right)^{1/p}.$$
Let $S_p^+$ be the set of all
nonnegative vectors $\bm{x}$ such that $\Vert\mathbf{x}\Vert_p=1$.
For $p\geq 1$, the $p$-spectral radius of an $r$-uniform hypergraph $H$ is
$$\rho_p(H)=\max_{\bm{x}\in S_p^+} P_H(\mathbf{x})= \max_{\bm{x}\in S_p^+}r\sum\limits_{\{i_1,  \ldots,  i_r\}\in E(H)} x_{i_1}\cdots x_{i_r}.
$$
Since $S_p^+$ is compact, the maximum can be always reached by some ${\bf x}\in S_p^+$,
which is called a {\em Perron} vector for $\rho_p(H)$. The Perron vector satisfies the following
eigen-equations:
\begin{equation}
  \label{eq:eigen}
  \sum_{\{i_2,\ldots, i_r\}\in E(H_i)}x_{i_2}\cdots x_{i_r}=\rho_p(H) x_i^{p-1} \mbox{ for any } x_i\ne 0. 
\end{equation}
Here $H_i$, {\em the link hypergraph of $H$ at $i$,}
consists of all $(r-1)$-tuple $f$ such that $f\cup\{i\}\in E(H)$.

It is known that the Perron vector is always positive and unique when $p>r$, or when $p=r$ and $H$ is connected. 
In general, a Perron vector could be neither unique nor positive. Since we didn't use the positivity and
uniqueness of the Perron vector in this paper, we will omit the detail of Perron-Frobenius Theorems
for hypergraphs. Readers are encouraged to read \cite{ChangPearsonZhang, FriedlandGaubetHan, YangYang}.

The $p$-spectral radius was introduced by Keevash-Lenz-Mubayi \cite{KLM2014} and followed by 
Nikiforov \cite{Nikiforov2014} in 2014. The $\rho_p(H)$ encompasses three important parameters of
$H$:
at $p=1$, $\frac{1}{r}\rho_1(H)$ is the {\em Lagrangian} of $H$ (Nikiforov referred it
as the MS-index of $H$ in honor of Motzkin and Straus);
at $p=r$, $\rho_r(H)$ is just the spectral radius $\rho(H)$; and at $p=\infty$,
$\frac{1}{r}\lim_{p\to\infty}\rho_p(H)$ is the number of edges in $H$.

The problem of determining the maximum of $\rho_p(H)$ among all $r$-uniform hypergraphs $H$
with a fixed number of edges
has a long history. For $p=r=2$,  Brualdi and Hoffman \cite{Brualdi} proved that the maximum of $\rho(H)$ among all graphs with ${k\choose 2}$ edges
is reached by the union of a complete graph on $k$ vertices and some possible isolated vertices. They conjectured that the maximum spectral radius of a graph $H$ with $m={s\choose 2} + t$ edges is attained by the graph
 $H_m$, which is obtained from complete graph $K_s$ by adding a new vertex and $t$ new edges.
 In 1987, Stanley \cite{stanley} proved that
 the spectral radius of a graph $H$ with $m$ edges
 is at most $\frac{\sqrt{1+8m}-1}{2}$. The equality holds if and only if $m={s\choose 2}$
 and $H$ is the union of the complete graph $K_s$ and some isolated vertices.
 Friedland \cite{Friedland} proved a bound which is tight on 
 the complete graph with one, two, or three edges removed or the complete graph with one edge added. 
 Rowlinson \cite{Row} finally confirmed Brualdi and Hoffman's conjecture, and proved that 
 $H_m$ attains the maximum spectral radius among all graphs with $m$ edges.

 For $r=2$ and $p=1$, Motzkin and Straus \cite{MS1965} proved that $\rho_1(H)=1-\frac{1}{\omega(H)}$,
 where $\omega(H)$ is the clique number of a graph $H$. For $r$-uniform hypergraph $H$, the value
 $\mu_r(H):=\frac{1}{r}\rho_1(H)$ is often referred as the {\em Lagrangian} of $H$,
 (or the MS index in \cite{Nikiforov2018}).  In 1989 Frankl and F\"{u}redi \cite{FF89}
 conjectured that the maximum Lagrangian of an $r$-graph $H$ with $m$ edges is realized
 by an $r$-uniform hypergraph consisting of the first $m$ sets in ${\mathbb{N} \choose r}$ in the
 colexicographic order (that is $A<B$ if $\max(A\Delta B)\in B$.) By Motzkin and Straus' result,
 the Frankl-F\"{u}redi conjecture holds for graphs. For $r=3$, Talbot \cite{Talbot02}
 and Tang-Peng-Zhang-Zhao \cite{TPZZ} confirmed the Frankl-F\"{u}redi conjecture for almost all values  of $m$. Tyomkyn \cite{Tyomkyn} proved the Frankl-F\"{u}redi conjecture holds almost everywhere
 for each $r\geq 4$.  Let
 \begin{equation}
   \label{eq:mum}
\mu_r(m)=\max\{ \mu(H)\colon H \mbox{ is an $r$-graph with $m$ edges}\}.   
 \end{equation}
 Note that the Frankl-F\"{u}redi conjecture does not provide an
 easy-to-use, closed-form expression for $\mu_r(m)$. Nikiforov \cite{Nikiforov2018}
 made the following conjecture:

 \begin{conjecture} \cite{Nikiforov2018}\label{c1}
   Let $r \geq 3$ and $H$ be an $r$-uniform hypergraph with $m$ edges. If $m = {s\choose r}$
   for some real $s$, then $\mu(H)\leq m s^{-r}.$
 \end{conjecture}

 The value of $\mu_r(m)$ conjectured by Frankl and F\"{u}redi is quite close to $ms^{-r}$,
 and moreover, both values coincide if $s$ is an integer. Tyomkyn \cite{Tyomkyn} called
 the case of integer $s$ the principal case of the Frankl-F\"{u}redi’s conjecture,
 and solved it for any $r \geq 4$ and $m$ sufficiently large. Talbot \cite{Talbot02}
 had solved the principal case for $r = 3$ and any $m$. Here we completely solved the principal case for all $r\geq 2$ and any $m$.

 The case $r=2$ of \autoref{c1} 
 is followed by  Motzkin and Straus' result. Nikiforov \cite{Nikiforov2018} proved
 this conjecture for $r=3,4,5$; and for the case $s\geq 4(r-1)(r-2)$.
 In this paper, we completely settled this conjecture
 for all $r\geq 2$ and $m\geq 0$. 
 
 \begin{theorem} \label{t1}
   Let $r \geq 2$ and $H$ be an $r$-uniform hypergraph with $m$ edges. Write $m = {s\choose r}$
   for some real $s\geq r-1$. We have $$\mu(H)\leq m s^{-r}.$$
   The equality holds if and only if $s$ is an integer and $H$ is the complete  $r$-uniform hypergraph
   $K^r_s$  possibly with some isolated vertices added.
 \end{theorem}
 
 Using the power-mean inequality, Nikiforov \cite{Nikiforov2014} show that
 $\left(\frac{\rho_p(H)}{rm}\right)^p$ is a non-increasing function of $p$.  
 Thus \autoref{t1} immediately implies the following theorem for general $p$.

\begin{theorem}\label{t2}
  For any integer $r\geq 2$ and any real $p\geq 1$, suppose that $H$ is an $r$-uniform hypergraph with $m={s\choose r}$ (for some real $s\geq r-1$) edges.  Then its $p$-spectral radius $\rho_p(H)$ satisfies
  \begin{equation}
    \label{eq:rhop}
\rho_p(H)\leq \frac{rm}{s^{r/p}}.    
  \end{equation}
The equality holds if and only if 
$s$ is an integer and $H$ is the complete $r$-uniform
hypergraph $K_s^r$ possibly with some isolated vertices added.
\end{theorem}

For $r=p=2$, it implies that $\rho_2(H)\leq \frac{\sqrt{8m+1}-1}{2}$, which is exactly
the Stanley's theorem.  Bai and Lu \cite{maxspectra} proved
the diagonal case  $p=r$ for all $r\geq 2$.
Nikiforov showed that the theorem holds for $r=2,3,4,5$ and any $p\geq 1$, and
when $m\geq {4(r-1)(r-2)\choose r}$.

Our method is very different from Nikiforov's approach --- the inequalities on symmetric functions.
We use double induction on both $r$ and $m$. The analytic method that we used here, is similar to the one used
by Bai and Lu \cite{maxspectra}.
The paper is organized as follows: In section 2, we prove several key lemmas on some special functions.
We prove \autoref{t1} and \autoref{t2} in the last section.

\section{Lemmas on important  functions}
For a fixed positive integer $r$, consider the polynomial $p_r(x)=\frac{x(x-1)\cdots (x-r+1)}{r!}.$
Since the binomial coefficient ${n\choose r}=p_r(n)$, we view ${x\choose r}$
as the polynomial $p_r(x)$.
Note that $p_r(x)$ is an increasing function over the interval $[r-1,\infty)$
so that the inverse function exists. Let $p^{-1}_r\colon [0,\infty) \to
[r-1, \infty)$ denote the inverse function of
$p_r(x)$ (when restricted to the interval $[r-1, \infty)$.

\begin{lemma}\label{l1}
  Let $u=p_r^{-1}(x)$ and $g_r(x)=x u^{-r}$.
We have
\begin{align}
  g_r'(x) &= u^{-r}(1- \frac{r}{u\sum_{i=0}^{r-1} \frac{1}{u-i}})
            =\frac{\sum_{i=0}^{r-1} \frac{i}{u-i}}{u^{r+1}\sum_{i=0}^{r-1} \frac{1}{u-i}};\\
  g_r''(x) &\leq -\frac{r+1}{u x \sum_{j=0}^{r-1} \frac{1}{u-j}}|g_r'(x)|.
  \end{align}
\end{lemma}

\begin{proof}
  Since $x={u\choose r}$, we have
  \begin{equation} 
 \ln x=\sum_{j=0}^{r-1}\ln(u-j) -\ln (r!).   
  \end{equation}
  Taking derivative, we have
  \begin{equation}
    \label{eq:du}
\frac{dx}{du}= x \cdot \frac{d}{du}\left[\sum_{j=0}^{r-1}\ln(u-j) -\ln (r!)\right]
= x \sum_{j=0}^{r-1} \frac{1}{u-j}.    
  \end{equation}

Thus,
\begin{align*}
g_r'(x)&=  \frac{d}{dx} (xu^{-r})\\
%  &= u^{-r} -rxu^{-(r+1)}\frac{d u}{dx}\\
       &= u^{-r} -ru^{-(r+1)}\frac{1}{\sum_{j=0}^{r-1} \frac{1}{u-j}}\\
       &=u^{-r}\left(1-r\frac{1}{u\sum_{j=0}^{r-1} \frac{1}{u-j}}\right)\\
  &=\frac{\sum_{i=0}^{r-1} \frac{i}{u-i}}{u^{r+1}\sum_{i=0}^{r-1} \frac{1}{u-i}}.
\end{align*}
Now we compute the second derivative. By the chain rule, we have
\begin{align*}
  g_r''(x)&= \frac{d}{du} (g_r'(x)) \cdot \frac{du}{dx}\\
         &= \left( -r u^{-r-1}+\frac{r(r+1)u^{-(r+2)}}{\sum_{j=0}^{r-1} \frac{1}{u-j}}
           -ru^{-(r+1)}\frac{\sum_{j=0}^{r-1} \frac{1}{(u-j)^2}}{\left(\sum_{j=0}^{r-1} \frac{1}{u-j}\right)^2} \right) \frac{du}{dx}\\
  &=-\frac{r}{u^{r+1}x \sum_{j=0}^{r-1} \frac{1}{u-j}} \left(1-\frac{r+1}{u\sum_{j=0}^{r-1} \frac{1}{u-j}}
           +\frac{\sum_{j=0}^{r-1} \frac{1}{(u-j)^2}}{\left(\sum_{j=0}^{r-1} \frac{1}{u-j}\right)^2} \right)\\
&=-\frac{r}{u^{r+1}x \sum_{j=0}^{r-1} \frac{1}{u-j}} \left(1- \frac{r}{u\sum_{j=0}^{r-1} \frac{1}{u-j}}
           +\frac{\sum_{j=0}^{r-1} \frac{j}{(u-j)^2}}{u\left(\sum_{j=0}^{r-1} \frac{1}{u-j}\right)^2} \right).
\end{align*}
Note that both sequences $a_j:=\frac{1}{u-j}$ and $b_j:=\frac{j}{u-j}$ are increasing sequence in $j$. Apply the Chebyshev's sum inequality: $r\sum_{j=0}^{r-1}a_jb_j\geq (\sum_{j=0}^{r-1}a_j)(\sum_{j=0}^{r-1}b_j)$.
We have
\begin{equation} \label{eq:Chebysev}
  r\sum_{j=0}^{r-1} \frac{j}{(u-j)^2}\geq
  \left(\sum_{j=0}^{r-1} \frac{1}{u-j}\right) \left( \sum_{j=0}^{r-1} \frac{j}{u-j}\right).
  \end{equation}
Applying Inequality \eqref{eq:Chebysev}, we get
  \begin{align*}
    -g_r''(x) % &=\frac{r}{u^{r+1}x \sum_{j=0}^{r-1} \frac{1}{u-j}} \left(1- \frac{r}{u\sum_{j=0}^{r-1} \frac{1}{u-j}}
%              +\frac{\sum_{j=0}^{r-1} \frac{j}{(u-j)^2}}{u\left(\sum_{j=0}^{r-1} \frac{1}{u-j}\right)^2} \right)\\
    &\geq \frac{r}{u^{r+1}x \sum_{j=0}^{r-1} \frac{1}{u-j}} \left(1- \frac{r}{u\sum_{j=0}^{r-1} \frac{1}{u-j}}
      +\frac{\sum_{j=0}^{r-1} \frac{j}{u-j}}{ru\sum_{j=0}^{r-1} \frac{1}{u-j}} \right)\\
             &=\frac{r}{u^{r+1}x \sum_{j=0}^{r-1} \frac{1}{u-j}}
               \left(1- \frac{r}{u\sum_{j=0}^{r-1} \frac{1}{u-j}} +\frac{1}{r}-\frac{1}{u\sum_{j=0}^{r-1} \frac{1}{u-j}}
               \right)\\
    &=\frac{r+1}{u^{r+1}x \sum_{j=0}^{r-1} \frac{1}{u-j}}
               \left(1- \frac{r}{u\sum_{j=0}^{r-1} \frac{1}{u-j}}\right)\\
    &=\frac{r+1}{u x \sum_{j=0}^{r-1} \frac{1}{u-j}}|g_r'(x)|.
  \end{align*}
  \end{proof}

For a fixed $r$ and $m={s\choose r}$, let $d_0={s-1\choose r-1}$.
We define two functions $A,B\colon [d_0, m]\to \mathbb{R}$ as
\begin{align}
  A(x)&:=\frac{x}{(p_{r-1}^{-1}(x))^{r-1}},\\
  B(x)&:=\frac{m-x}{(p_{r}^{-1}(m-x))^{r}}.  
  \end{align}
  Write $x={t\choose r-1}$ and $m-x={u\choose r}$. We have
  \begin{align}
    \label{eq:A}
    A(x)& =xt^{-(r-1)},\\
    \label{eq:B}
    B(x)&=(m-x)u^{-r}.
  \end{align}
   Applying \autoref{l1}, we have
   \begin{align}
     \label{eq:A'}
     A'(x)&=t^{-(r-1)}\left(1-\frac{(r-1)}{t\sum_{i=0}^{r-2}\frac{1}{t-i}} \right)
            =\frac{\sum_{i=0}^{r-2}\frac{i}{t-i}}{t^r\sum_{i=0}^{r-2}\frac{1}{t-i}}
            ,\\
     \label{eq:B'}
     B'(x)&= -u^{-r}\left(1-\frac{r}{u\sum_{i=0}^{r-1}\frac{1}{u-i}}\right)
            =-\frac{\sum_{i=0}^{r-1}\frac{i}{u-i}}{u^{r+1}\sum_{i=0}^{r-1}\frac{1}{u-i}}
            ,\\
     \label{eq:A''}
     A''(x)&< - \frac{r}{t x\sum_{i=0}^{r-2}\frac{1}{t-i}} |A'(x)|,\\
     \label{eq:B''}
     B''(x)&<-\frac{r+1}{u(m-x) \sum_{i=0}^{r-1} \frac{1}{u-i}} |B'(x)|.
  \end{align}
  
We have the following lemma.

\begin{lemma} \label{l2} We have
  \begin{align}
     \label{eq:tsu} 
    t&\geq s-1\geq u,\\
     \label{eq:A-rB}
    A(x)- rB(x)&\geq A(x)\frac{r-1}{u},\\
    A(x)&\geq \frac{ru}{u-r+1} B(x).
    \label{eq:A/B}      
  \end{align}
   The equalities hold if and only if $x=d_0$.
\end{lemma} 
\begin{proof}
    Since $x\geq d_0$, we have ${t\choose r-1}\geq {s-1\choose r-1}$. It implies
  $t\geq s-1$ with the equality holds if and only if $d=d_0$. We also have
  \begin{equation}
    \label{eq:ur}
  {u\choose r}=m-x\leq m-d_0={s\choose r}-  {s-1\choose r-1}={s-1\choose r}.  
  \end{equation}
  Thus $u\leq s-1$ with the equality holds if and only $x=d_0$.
 Since $t\geq u$, we have
  \begin{align*}
    \frac{A(x)-rB(x)}{A(x)}&= 1- \frac{r{u\choose r} u^{-r}}{{t\choose r-1} t^{-(r-1)}}\\
                      &= 1-\frac{\prod_{i=0}^{r-1}(1-\frac{i}{u})}{\prod_{i=0}^{r-2}(1-\frac{i}{t})}\\
                           &\geq 1 - (1-\frac{r-1}{u})\\
    &=\frac{r-1}{u}.
  \end{align*}
 The equality holds if and only if $t=u$ (or $x=d_0$).
 Inequality \eqref{eq:A/B} can be derived from \eqref{eq:A-rB} by solving $A(x)$.
\end{proof}

\begin{lemma} \label{l3}
  We have
  \begin{align}
    \label{eq:ut} 
    \frac{u\sum_{i=0}^{r-1}\frac{1}{u-i}}{r}&> \frac{t\sum_{i=0}^{r-2}\frac{1}{t-i}}{r-1},\\
    \label{eq:A'B'}
      -u^rB'(x)&>t^{r-1}A'(x).
    \end{align}
  \end{lemma}
\begin{proof}
  Observe $\frac{u}{u-i}\geq \frac{t}{t-i}$ since $t\geq u$.
  It suffices to show
  \begin{equation}
    \label{eq:t1}
    \frac{t\sum_{i=0}^{r-1}\frac{1}{t-i}}{r} > \frac{t\sum_{i=0}^{r-2}\frac{1}{t-i}}{r-1}.
  \end{equation}
  Equivalently, 
  \begin{equation}
    \label{eq:t2}
 (r-1)\frac{1}{t-r+1}> \sum_{i=0}^{r-2}\frac{1}{t-i},   
  \end{equation}
which holds since $\frac{1}{t-r+1}>\frac{1}{t-i}$ for all $0\leq i<r-2$.

For Inequality \eqref{eq:A'B'}, we have
\begin{align*}
  -u^rB'(x)&=1-\frac{r}{u\sum_{i=0}^{r-1}\frac{1}{u-i}}\\
  &>1-\frac{(r-1)}{t\sum_{i=0}^{r-2}\frac{1}{t-i}}\\
  &=t^{r-1}A'(x).
\end{align*}
\end{proof}

The following lemma will play the key role in our proof of \autoref{t1}.
\begin{lemma}\label{key}
  For any integers $r\geq 2$, $m\geq 1$, and any real $x\in [d_0,m)$, we have
  \begin{equation}
    \label{eq:decreasing}
      (A(x)-rB(x))A'(x)+(r-1)A(x)B'(x)<0.  
  \end{equation}
\end{lemma}
\begin{proof}
  Let $F(x):=(A(x)-rB(x))A'(x)+(r-1)A(x)B'(x)$.
  We claim:
  \begin{description}
  \item[Claim a:] $F(d_0)<0$.
  \item[Claim b:] $F'(x)<0$ for all $x\in [d_0,m)$.
  \end{description}
These claims implies $F(d)\leq F(d_0)<0$.

\noindent {\bf Proof of Claim a:}
At $d=d_0$, we have $t=s-1=u$ by \autoref{l2}.
We have
\begin{align*}
  F(d_0)&=(A(d_0)-rB(d_0))A'(d_0)+(r-1)A(d_0)B'(d_0) \\
        &= \frac{r-1}{s-1}A(d_0)A'(d_0)+(r-1)A(d_0)B'(d_0) \hspace*{1cm}
          \mbox{by \autoref{l2}}\\
        &=  \frac{r-1}{(s-1)^r}A(d_0)\left((s-1)^{r-1}A(d_0)+ (s-1)^r B'(d_0) \right) \\
        &=  \frac{r-1}{(s-1)^r}A(d_0)\left(t^{r-1}A(d_0)+ u^r B'(d_0) \right) \hspace*{1cm}\mbox{ since $t=s-1=u$}\\        &<0 \hspace*{6.5cm} \mbox{ by \eqref{eq:A'B'}}.
\end{align*}
The proof of Claim a is finished.

\noindent
{\bf Proof of Claim b:}
We have
\begin{align} \nonumber
  F'(x)&= (A(x)-rB(x))A''(x)+ (A'(x)-rB'(x))A'(x)\\
 \nonumber  &\;\; + (r-1)A'(x)B'(x)+(r-1)A(x)B''(x)\\
       &=(A(x)-rB(x))A''(x) +(A'(x))^2 -A'(x)B'(x)+ (r-1)A(x)B''(x).
         \label{eq:F'x}
\end{align}
Note that $A(x)$, $B(x)$, and $A'(x)$ are positive and $B'(x)$, $A''(x)$, and $B''(x)$ are negative. By Inequality \eqref{eq:B''}, we have
\begin{equation} \label{eq:AB''1}
  -(r-1)A(x)B''(x)>A(x)\frac{(r-1)(r+1)}{u(m-x)\sum_{i=0}^{r-1}\frac{1}{u-i}} |B'(x)|.
\end{equation}
Combining Inequality \eqref{eq:AB''1}, \eqref{eq:A/B}, and Equation \eqref{eq:B},  we get
\begin{equation}\label{eq: 25}
  (r-1)A(x)|B''(x)| > |B'(x)| \frac{r^3-r}{u^r (u-r+1) \sum_{i=0}^{r-1}\frac{1}{u-i}}. 
 \end{equation} 
Thus, we have
\begin{align*}
  \frac{ (r-1)A(x) |B''(x)|}{A'(x)|B'(x)|}
  &> \frac{1}{A'(x)} \frac{r^3-r}{u^r (u-r+1) \sum_{i=0}^{r-1}\frac{1}{u-i}}\\
  &= \frac{t^r \sum_{i=0}^{r-2}\frac{1}{t-i}}{\sum_{i=0}^{r-2}\frac{i}{t-i}}\frac{(r^2-1)r}{u^r (u-r+1) \sum_{i=0}^{r-1}\frac{1}{u-i}}  \hspace*{1cm} \eqref{eq:A'}
  \\
  &>\frac{t^r(r^2-1)(r-1)}{u^r\sum_{i=0}^{r-2}\frac{i (u-r+1)}{t-i}}\hspace*{2cm} \mbox{ by \eqref{eq:ut}}\\
  &> \frac{t^r(r^2-1)(r-1)}{u^r\sum_{i=0}^{r-2}i}\hspace*{1cm}
    \mbox{since $u-r+1\leq t-i$} \\
  &=\frac{2(r^2-1)t^r}{(r-2) u^r}\\
  &>2(r+1)\frac{t^r}{u^r}.
\end{align*}
Hence,
\begin{align}
  \nonumber
  (r-1)A(x) |B''(x)|&>2(r+1)\frac{t^r}{u^r} A'(x) |B'(x)|\\
                    &>2r\frac{t^r}{u^r} A'(x) |B'(x)| +2 A'(x) |B'(x)|
                      \hspace*{1cm} \mbox{ since } t\geq u
                       \nonumber \\
                    &>\frac{2r t^{2r-1}}{u^{2r}} (A'(x))^2  +2 A'(x) |B'(x)|
                      \hspace*{1cm} \mbox{ by \eqref{eq:A'B'}} \nonumber\\
                      &>\frac{2r}{t} (A'(x))^2  +2 A'(x) |B'(x)|
                      \hspace*{5mm} \mbox{ since } t\geq u. \label{eq:AB''}
\end{align}
Now we estimate the lower bound of $(A(x)-rB(x))|A''(x)|$. 
Applying Inequalities \eqref{eq:A-rB} and \eqref{eq:A''}, we have
\begin{equation}
  \label{eq:28}
  (A(x)-rB(x))|A''(x)|>\frac{r-1}{u} A(x) \frac{r}{xt \sum_{i=0}^{r-2}\frac{1}{t-i}} A'(x).
\end{equation}
Combining with Equation \eqref{eq:A} and $t\geq u$, we get
\begin{equation}
  \label{eq:29}
  (A(x)-rB(x))|A''(x)|>\frac{r(r-1)}{t^{r+1} \sum_{i=0}^{r-2}\frac{1}{t-i}} A'(x).
\end{equation}
We have
\begin{align}
  \nonumber
  (A(x)-rB(x))|A''(x)| &>\frac{r(r-1)}{t^{r+1} \sum_{i=0}^{r-2}\frac{1}{t-i}} A'(x)\\ 
  \nonumber             &=(A'(x))^2 \frac{r(r-1)}{t\sum_{i=0}^{r-2}\frac{i}{t-i}} \hspace*{1cm}
                         \mbox{ by } \eqref{eq:A'}
           \\
  \nonumber           &>(A'(x))^2 \frac{r(r-1)}{t\sum_{i=0}^{r-2}\frac{i}{t-r+2}}
\hspace*{1cm}
                         \mbox{ since } t-i\geq t-r+2
  \\
  \label{eq:BA''}           &>(A'(x))^2 \frac{2(t-r+2)}{t}.
\end{align}
Combining Inequalities \eqref{eq:AB''} and \eqref{eq:BA''}, we get
\begin{align*}
  (A(x)-rB(x))|A''(x)| &+(r-1)A(x) |B''(x)|\\
  &>(A'(x))^2 \frac{2(t-r+2)}{t}
                                             + \frac{2r}{t} (A'(x))^2  +2 A'(x) |B'(x)|\\
                                           &=(A'(x))^2 \frac{2(t+2)}{t}  +2 A'(x) |B'(x)|\\
  &>2(A'(x))^2 +2 A'(x) |B'(x)|.                                            
\end{align*}
Recall that $A(x)$, $B(x)$, and $A'(x)$ are positive and $B'(x)$, $A''(x)$, and $B''(x)$ are negative. This implies
\begin{align*}
  F'(x)&=(A(x)-rB(x))A''(x) +(A'(x))^2 -A'(x)B'(x)+ (r-1)A(x)B''(x)\\
       &<-(A'(x))^2 - A'(x)|B'(x)|\\
       &<0.
\end{align*}
We finished the proof of Claim b.
\end{proof}

\section{Proof of  \autoref{t1} and \autoref{t2}}
For fixed $r\geq 2$ and $p\geq 1$,
let $H=(V, E)$ be an $r$-uniform hypergraph whose $p$-spectral radius attains the maximum among all the $r$-uniform hypergraphs with $m$ edges. We call $H$ a {\em $p$-maximum} hypergraph.
         %          Let ${\bf x}\in S_p^+$ be a Perron corresponding to $\rho_p(H)$.
We have 

\begin{lemma}\label{degree}
  For any $r\geq 2$, $p\geq 1$, $m\geq 0$,  there always exists  a $p$-maximum hypergraph $H$
  (with $m$ edges) satisfying
  \begin{quotation}
    ``There exists a vertex $v$ with degree $d_v\geq d_0:= \frac{rm}{p_r^{-1}(m)}$ so that
    some Perron vector achieves the maximum at $v$.''
  \end{quotation}
 \end{lemma}
To prove this lemma, we will use  Lov\'asz's theorem on the shadow set.

\begin{definition}\label{shadef}
  Given a family $\mathcal{F}$ of $r$-sets, the shadow $\partial(\mathcal{F})$ is defined as
  \begin{equation}
    \label{eq:shadow}
\partial(\mathcal{F})=\{e': e'=e\setminus\{v\}, \text{for some $e\in\mathcal{F}$, and $v\in e$} \}.    
  \end{equation}
\end{definition}

Here is Lov\'asz's theorem on the shadow sets, which is slightly weaker but more convenient to use
than Kruskal-Katona's Theorem \cite{Katona, Kruskal}.
\begin{theorem}\label{shadow} 
(Lov\'asz \cite{Lovas}) Any $r$-uniform set family $\mathcal{F}$ of size $m =
{x\choose r}$ where $x$ is a real and $x \geq r$, must have
\begin{equation}
  \label{eq:lovasz}
|\partial(F)|\geq {x \choose r-1}.  
\end{equation}
\end{theorem}

Let $d_H(v)$ the degree of the vertex $v$ in $H$. 
Let $H-v$ be the induced subgraph obtained from $H$ by deleting the vertex $v$. 
Let $H_v$ be the link hypergraph of $v$; i.e., it contains all $(r-1)$-tuple $f$ such that
$f\cup \{v\}$ is an edge of $H$.
By definition of $H_v$, $d(v)$ is also the number of edges in $H_v$.

\begin{proof}[Proof of Lemma \ref{degree}:]
  Starting with any $p$-maximum hypegraph $H^0$ (of $m$ edges),
  let ${\bf x}$ be a Perron vector of $H^0$.
  Without loss of generality we can assume $x_1\geq x_2\geq \cdots \geq x_n\geq 0$.
  Let $v=v_1$ be the vertex corresponding to the first entry $x_1$.

If $\partial(H^0-v)$ is a subgraph of $H^0_v$, then we let $H=H^0$.
Otherwise, there exists an $(r-1)$-subset 
$\{v_{i_1},\ldots, v_{i_{r-1}}\}$ in $\partial(H^0-v)$ but not in $E(H^0_v)$.
By the definition of the shadow set $\partial(H^0-v)$, there is a vertex $u\not= v$
so that $\{u,v_{i_1},\ldots, v_{i_{r-1}}\}$ is an edge of $H^0-v$. By moving this edge
from $u$ to $v$, we obtain a new hypergraph $H^1$ from $H^0$, which still has $m$ edges.
Note that
\begin{equation}
  \label{eq:seq}
\rho_p(H^1)\geq P_{H^1}(\bm{x})\geq P_{H^0}(\bm{x})=\rho_p(H^0).  
\end{equation}
Since $H^0$ is a maximum hypergraph, we have $\rho_p(H^0)\geq \rho_p(H^1).$
This forces all inequalities in \eqref{eq:seq} to be equal.
In particular, $H^1$ is also a $p$-maximum hypergraph and $\bf x$ is 
a Perron vector for $H^1$.  Since the number of hypergraphs with $m$
edges is finite, we may continue this process until we reach a hypergraph $H$ satisfying
$\partial(H-v)\subseteq H_v$.

Write $m={s\choose r}$
and $|E(H-v)|={u \choose r}$ for some real numbers $s, u\geq r-1$.
By Theorem \ref{shadow}, we have
\begin{equation}
  \label{eq:shadow1}
|\partial(H-v)|\geq {u \choose r-1}.  
\end{equation}
We have 
$$|E(H_v)|\geq |\partial(H-v)|\geq
 {u \choose r-1}.$$ 
Thus,
\begin{align*}
{s\choose r}&=|E(H)| \\
&=|E(H_v)|+|E(H-v)|\\
&\geq  {u \choose r-1} +{u \choose r}\\
&={u+1\choose r}. 
\end{align*}
Thus, $s\geq u+1$. It implies 
\begin{align*}
|E(H_v)|&=e-|E(H-v)|\\
&={s\choose r}-{u \choose r} \\
&\geq {s\choose r}-{s-1\choose r}\\
        &={s-1\choose r-1}\\
&=\frac{rm}{s}\\
&=d_0.
\end{align*}
The proof is finished.
\end{proof}

\begin{proof}[Proof of \autoref{t1}]
  Note that $\mu(H)=\frac{1}{r}\rho_1(H)$. It is sufficient to prove
  for any $r$-uniform hypergraph $H$ with $m$ edges
  \begin{equation}
    \label{eq:rho1}
    \rho_1(H)\leq \frac{rm}{(p_r^{-1}(m))^r}.
  \end{equation}  
  We will use double inductions on $r$ and $m$ to prove Inequality \eqref{eq:rho1}.
  The assertion holds trivially for $r=1$ and any $m$.

Inductively, we assume the statement is true for all $(r-1)$-hypergraphs. For $r$-hypergraph, clearly, the statement is trivial for the cases $m=0, 1$. We only need to consider $m\geq 2$. We assume the statement holds for all
$r$-hypergraphs with less than $m$ edges. 
Let $H$ be the maximum hypergraph guaranteed by Lemma \ref{degree} (for $p=1$).
Then $H$ has a Perron vector ${\bf x}$ and a vertex $v$ so that
$\bf x$ reaches the maximum at $v$ and $d_v\geq d_0$.
Without loss of generality, we assume $v=v_1$.
Write $m={s\choose r}$, $d_v={t\choose r-1}$, and $m-d_v={u\choose r}$.
Then $d_0={s-1\choose r-1}$.

Recall that $H-v$ is the induced hypergraph obtained from $H$ by deleting the vertex $v$ and
$H_v$ is the link graph of $H$ at $v$. Let
${\bf y}=(y_2,\ldots, y_n)\in \mathbb{R}^{n-1}$ so that $y_j=x_j/(1-x_1)$ for all $j\geq 2$.
On one hand, from Eigen-equation \eqref{eq:eigen}, we have
\begin{align*}
  \rho_1(H)&=\sum_{\{i_2,\ldots, i_r\}\in E(H_v)}x_{i_2}\cdots x_{i_r}\\
  &=(1-x_1)^{r-1} \sum_{\{i_2\ldots, i_r\}\in E(H_v)}y_{i_2}\cdots y_{i_r}\\
  &\leq \frac{1}{r-1} (1-x_1)^{r-1}\rho_1(H_v).  
\end{align*}
By inductive hypothesis, we have $\rho_1(H_v)\leq (r-1) d_v \left( p_{r-1}^{-1}(d_v) \right)^{-(r-1)}$.
Thus,
\begin{equation}
  \label{eq:G1}
  \rho_1(H)\leq (1-x_1)^{r-1} d_v \cdot \left( p_{r-1}^{-1}(d_v) \right)^{-(r-1)}.
\end{equation}

On the other hand, we have
\begin{align}\nonumber
  (1-rx_1) \rho_1(H) &=r \sum_{\{i_1,i_2,\ldots, i_r\}\in E(H)} x_{i_1}x_{i_2}\cdots x_{i_r}
                       - r x_1 \sum_{\{i_2\,\ldots, i_r\} \in E(G_1)} x_{i_2}\cdots x_{i_r} \\
  \label{eq:38}
                     &= r \sum_{\{i_1,i_2,\cdots, i_r\} \in E(H-v)} x_{i_1}x_{i_2}\cdots x_{i_r}\\
  \nonumber
   &= (1-x_1)^r \cdot r\sum_{\{i_1,i_2,\ldots, i_r\}\in E(H-v)} y_{i_1}y_{i_2}\cdots y_{i_r}\\
                     &\leq (1-x_1)^r \rho_1(H-v).
\end{align}
First by Equation \eqref{eq:38}, we have $(1-rx_1) \rho_1(H)\geq 0$. Thus $x_1\leq \frac{1}{r}$.
Second by inductive hypothesis, we have $\rho_1(H-v)\leq r (m-d_v) \cdot \left( p_{r}^{-1}(m-d_v) \right)^{-r}$.
Thus,
\begin{equation}
  \label{eq:H1}
  \rho_1(H)\leq \frac{(1-x_1)^{r}}{1-rx_1} r (m-d_v) \left( p_{r}^{-1}(m-d_v) \right)^{-r}.
\end{equation}

Let $A(x):=x \cdot \left( p_{r-1}^{-1}(x) \right)^{-(r-1)}$ and $B(x):=(m-x) \left( p_{r}^{-1}(m-x) \right)^{-r}$. Then
$$\rho_1(H)\leq \min_{x_1\in [0,\frac{1}{r}]} \left\{ A(d_v) (1-x_1)^{r-1}, B(d_v) r \frac{(1-x_1)^{r}}{1-rx_1} \right\}.$$
As the function of $x_1$, the first function decreases and the second
function increases. Two functions intersect at the point
$x_0=(A(d_v)-rB(d_v))/(rA(d_v)-rB(d_v))$. We have
\begin{equation}
  \label{eq:h}
  \rho_1(H)\leq  A(d_v) (1-x_0)^{r-1}= \frac{(r-1)^{r-1} A(d_v)^r}{r^{r-1} (A(d_v)-B(d_v))^{r-1}}
  :=h(d_v).
\end{equation}
Here $h(x):=\frac{(r-1)^{r-1} A(x)^r}{r^{r-1} (A(x)-B(x))^{r-1}}.$
Then
\begin{equation}
  \label{eq:lnh}
\ln (h(x))=(r-1)(\ln(r-1)-\ln r)+ r\ln (A(x)) -(r-1) \ln (A(x)-B(x)).
\end{equation}
Taking derivative, we get
\begin{align*}
  \frac{h'(x)}{h(x)}&=\frac{rA'(x)}{A(x)}-(r-1)\frac{A'(x)-B'(x)}{A(x)-B(x)}\\
  &=\frac{(A(x)-rB(x))A'(x)+(r-1)A(x)B'(x))}{A(x)(A(x)-B(x))}.
\end{align*}
By \autoref{l2} and \autoref{key}, we have $h'(x)<0$ for all $x\in [d_0,m)$.
Thus $h(x)$ is a decreasing function. We have
\begin{equation}
  \label{eq:hdec}
h(d_v)\leq h(d_0).  
\end{equation}
Note that
\begin{align*}
  A(d_0)&=d_0(p_{r-1}^{-1}(d_0))^{-(r-1)}\\
        &={s-1\choose r-1}  (s-1)^{-(r-1)}\\
  &=\frac{\prod_{i=0}^{r-2}(1-\frac{i}{s-1})}{(r-1)!}.          
\end{align*}
\begin{align*}
  B(d_0)&=(m-x)(p_{r}^{-1}(m-d_0))^{-r}\\
        &={s-1\choose r}  (s-1)^{-r}\\
  &=\frac{\prod_{i=0}^{r-1}(1-\frac{i}{s-1})}{r!}. 
\end{align*}
Thus,
\begin{align*}
  A(d_0)-B(d_0) &=A(d_0)\left(1-\frac{1}{r}\left(1-\frac{r-1}{s-1}\right)\right)\\
  &=A(d_0)\frac{(r-1)s}{r(s-1)}.
\end{align*}
Plugging into $h(d_0)$, we get
\begin{align*}
  h(d_0) &=\frac{(r-1)^{r-1} A(d_0)^r}{r^{r-1} (A(d_0)-B(d_0))^{r-1}}\\
            &=\frac{(r-1)^{r-1}}{r^{r-1}} A(d_0) \left(\frac{A(d_0)}{A(d_0)-B(d_0)} \right)^{r-1}\\
            &=\frac{(r-1)^{r-1}}{r^{r-1}} \frac{\prod_{i=0}^{r-2}(1-\frac{i}{s-1})}{(r-1)!}
              \left(\frac{r(s-1)}{(r-1)s}  \right)^{r-1}\\
            &=\frac{\prod_{i=0}^{r-2}(1-\frac{1+i}{s})}{(r-1)!}\\
            &=\frac{r{s\choose r}}{s^r}.\\
            &=\frac{rm}{(p_r^{-1}(m))^r}.
\end{align*}

Combining it with Inqualities \eqref{eq:h} and \eqref{eq:hdec}, we get
\begin{equation}
  \rho_1(H)\leq h(d_v)\leq h(d_0)\leq \frac{rm}{(p_r^{-1}(m))^r}.
\end{equation}
The inductive proof of Inequality \eqref{eq:rho1} is finished.

When the inequality holds, we must have $m={s\choose r}$, $|E(H_v)|=d_v={s-1\choose r-1}$, $\rho_1(H_v)={s-1 \choose r-1} (s-1)^{-(r-1)}$,
  $|E(H-v)|={s-1\choose r}$, $\rho_1(H-v)={s-1 \choose r} (s-1)^{-r}$, and $\partial(H-v)\subseteq H_v$.
By inductive hypothesis, $H_v$ is the
complete graph $K_{s-1}^{r-1}$ and $H-v$ is the complete graph $K_{s-1}^r$.
Together with $\partial(H-v)\subseteq H_v$, we conclude that $H$ is the complete graph $K_s^r$
and finished the inductive proof.

Since adding isolated vertices will not change the number of edges and the spectral radius,
 the inequality in \autoref{t1} holds if and only if $H$ is 
the complete hypergraph possibly with some isolated vertices added. 
\end{proof}

The following Lemma is due to Nikiforov \cite{Nikiforov2014}. Here we relay his proof for the completeness.
\begin{lemma}\cite{Nikiforov2014} \label{l4}
  For any $p\geq 1$ and any $r$-uniform hypergraph $H$ with $m$ edges, we have
  \begin{equation}
    \label{eq:1top}
    \rho_p(H)\leq \rho_1(H)^{1/p} (rm)^{1-1/p}.
  \end{equation}
\end{lemma}
\begin{proof}
  Let $\bm{x}$ be the Perron vector for $\rho_p(H)$. Then we have
  \begin{align*}
    \rho_p(H)&=r \sum_{\{i_1,\ldots, i_r\}\in E(H)} x_{i_1}\ldots x_{i_r}\\
             & \leq (rm)^{1-1/p}(r\sum_{\{i_1,\ldots, i_r\}\in E(H)} x_{i_1}^p\ldots x_{i_r}^p)^{1/p}\\
             &\leq  (rm)^{1-1/p} (\rho_1(H))^{1/p}.
  \end{align*}
\end{proof}

\begin{proof}[Proof of \autoref{t2}:]
  By  \autoref{t1}, we get
  \begin{equation}
    \label{eq:rho1H}
\rho_1(H)\leq \frac{rm}{s^{r}}.    
  \end{equation}
 Combining Inequalities \eqref{eq:1top} and \eqref{eq:rho1H},
 we get
 $$\rho_p(H)\leq \frac{rm}{s^{r/p}}.$$
\end{proof}

\end{document}